\documentclass[reqno, 12pt]{amsart}
\usepackage{amsmath,amsxtra,amssymb,latexsym, amscd,amsthm}
\usepackage[utf8]{inputenc}
\usepackage[mathscr]{euscript}
\usepackage{mathrsfs}
\usepackage{enumerate}
\usepackage{color}
\usepackage{cite}

\newtheorem {theorem}{Theorem}[section]
\newtheorem {corollary}{Corollary}[section]
\newtheorem {lemma}{Lemma}[section]

\newtheorem {proposition}{Proposition}[section]

\theoremstyle{definition}
\newtheorem{definition}{Definition}[section]
\newtheorem{example}{Example}[section]
\newtheorem{remark}{Remark}[section]

\def\ar{a\kern-.370em\raise.16ex\hbox{\char95\kern-0.53ex\char'47}\kern.05em}
\def\ees{{\accent"5E e}\kern-.385em\raise.2ex\hbox{\char'23}\kern-.08em}
\def\eex{{\accent"5E e}\kern-.470em\raise.3ex\hbox{\char'176}}
\def\AR{A\kern-.46em\raise.80ex\hbox{\char95\kern-0.53ex\char'47}\kern.13em}
\def\EES{{\accent"5E E}\kern-.5em\raise.8ex\hbox{\char'23 }}
\def\EEX{{\accent"5E E}\kern-.60em\raise.9ex\hbox{\char'176}\kern.1em}
\def\ow{o\kern-.42em\raise.82ex\hbox{
   \vrule width .12em height .0ex depth .075ex \kern-0.16em \char'56}\kern-.07em}
\def\OW{O\kern-.460em\raise1.36ex\hbox{
\vrule width .13em height .0ex depth .075ex \kern-0.16em \char'56}\kern-.07em}
\def\UW{U\kern-.42em\raise1.36ex\hbox{
\vrule width .13em height .0ex depth .075ex \kern-0.16em \char'56}\kern-.07em}
\def\DD{D\kern-.7em\raise0.4ex\hbox{\char '55}\kern.33em}
\newcommand*{\B}{\mathbb{B}}
\newcommand*{\R}{\mathbb{R}}
\newcommand*{\Z}{\mathbb{Z}}

\title[]{Some classical analysis results for continuous definable mappings}

\author{TRUONG XUAN DUC HA$^\dag$}
\address[Truong Xuan Duc Ha$^\dag$]{Thanglong Institute of Mathematics and Applied Sciences, Thanglong University, Hanoi, Vietnam}
\email{txdha@math.ac.vn}

\author{TI\EES N-S\OW N PH\d{A}M$^{\ddag}$}
\address[Ti\ees n-S\ow n Ph\d{a}m$^{\ddag}$]{Department of Mathematics, Dalat University, 1 Phu Dong Thien Vuong, Dalat, Vietnam}
\email{sonpt@dlu.edu.vn}

\date{ \today}


\begin{document}

\maketitle

\begin{abstract} 
In this paper, we show that some fundamental results for smooth mappings (e.g., the Brouwer degree formula, the implicit  function and inverse function theorems, the mean value theorem, Sard's theorem, Hadamard's global invertibility criteria, Pourciau's surjectivity and openness results) have natural extensions for  continuous mappings that are definable in o-minimal structures. The arguments rely on nice properties of definable mappings and sets.
\end{abstract}

\section{Introduction}
It is well-known that classical theorems for smooth mappings such as {\em the implicit function and inverse function theorems, the mean value theorem, $\ldots$} are basic results in real analysis that play a crucial role in a variety of applications including optimization, electrical network theory, theory of economic equilibria, stability of nonlinear systems. Such results have received a huge amount of attention over the years. Variations and corollaries have been discovered and extensions to non-smooth mappings have been established. In this paper we have cited, essentially, only few relevant contributions. For more details, we refer the reader to the comprehensive monographs \cite{Clarke1990, Dontchev2009, Jeyakumar2008, Klatte2002, Krantz2013} with the references therein.  
 
The purpose of the present paper is to generalize some classical theorems for continuous mappings which are definable in o-minimal structures (see Section~\ref{Section3} for the definitions). To be more specific, consider a continuous definable mapping $f \colon  \Omega \to \R^n$, where $\Omega$ is an open definable set in $\R^n.$  For this mapping, the derivative mapping $df(x)  \colon \mathbb{R}^n \to \mathbb{R}^n$ may not exist everywhere in $\Omega,$ but one can at least assign to each point $x \in \Omega$ a certain value $\nu_f(x) \in [0, +\infty].$ Then we say that $x$ is a {\em regular point} of $f$ if $\nu_f(x) > 0.$ Now this notion of regularity has many nice properties as we shall see, but only the following fact needs to be singled out here: $x$ is a regular point of $f$ if and only if the mapping $df(x)$ is surjective whenever $f$ is continuously differentiable on a neighbourhood of $x.$ Results involving the functions $\nu_f(\cdot)$ therefore extend their continuously differentiable counterparts. Our proofs are quite elementary, using some basic facts in the theory of o-minimal structures.

We should mention at this point that definability is a fundamental concept of the theory of o-minimal structures (see, e.g., \cite{Coste2000, Dries1996, Dries1998}) which is also being very actively developed in the eighties of the twentieth century. The important point about definable mappings and sets is that they enjoy many nice properties and that they are void of ``pathologies'' so typical for generic objects of non-smooth analysis. As mentioned by Ioffe \cite{Ioffe2007}, definable mappings and sets look like an almost ideal playground for applicable finite dimensional variational analysis. So our study would provide some tools for further applications of definable mappings in non-smooth analysis.

The paper is organized as follows. Section~\ref{SectionPreliminary} contains preliminaries about definable mappings and sets. In Section~\ref{Section3}, the notions of regular and critical points and values  for continuous definable mappings are given and a non-smooth version of the Sard theorem is established. In Section~\ref{Section4} we formulate the mean value theorem, the inverse function and implicit function theorems for continuous definable mappings. In Section~\ref{Section5}, we show that the sign of the determinant of the derivative enjoys a stability property around regular points and can be used to calculate the Brouwer degree. Section~\ref{Section6} concerns with some global properties of continuous definable mappings; in particular, we extend some results of Hadamard~\cite{Hadamard1906} and Pourciau~\cite{Pourciau1983-1}  by providing sufficient conditions for such a mapping to be a global homeomorphism and to be surjective and quasi-interior.

\section{Preliminaries} \label{SectionPreliminary}

\subsection{Notation} 
In this paper,  we deal with the Euclidean space $\mathbb{R}^n$ equipped with the usual scalar product $\langle \cdot, \cdot \rangle$ and the corresponding Euclidean norm $\| \cdot\|.$ If $x \in \mathbb{R}^n$ and $\Omega \subset \mathbb{R}^n$ is a non-empty set then the distance from $x$ to $\Omega$ is defined by $\mathrm{dist}(x, \Omega):=\inf_{y \in \Omega}\|x - y\|.$ We denote by $\mathbb{B}_r(x)$ the open ball centered at $x$ with radius $r.$ The closure, boundary, convex hull, and closed convex hull of a set $\Omega \subset \mathbb{R}^n$ will be written as $\overline{\Omega},$ $\partial \Omega,$ ${\mathrm{co}}\Omega,$ and $\overline{\mathrm{co}}\Omega,$ respectively. 

We denote by $\mathbb{R}^{n\times n}$ the space of all linear mappings from $\R^n$ to $\R^n,$ which can be regarded as $n\times n$-matrices. The space $\mathbb{R}^{n\times n}$ is endowed with its usual matrix norm. For any $A \in \R^{n\times n},$ its {\em co-norm} is defined by
$$\sigma(A) :=\inf_{x \in \R^n,\,  \|x\|=1}\|Ax\|.$$
It is clear that $\sigma (A)>0$ if and only if $A$ is surjective. For further properties of the co-norm, see \cite{Kurdyka2000-1, Pourciau1988, Rabier1997}.

\subsection{O-minimal structures and definable mappings}

The notion of o-minimality was developed in the late 1980s after it was noticed that many proofs of analytic and geometric properties of semi-algebraic sets and mappings could be carried over verbatim for sub-analytic sets and mappings. We refer the reader to \cite{Coste2000, Dries1998, Dries1996} for the basic properties of o-minimal structures used in this paper. 

\begin{definition}{\rm
An {\em o-minimal structure} on $(\mathbb{R}, +, \cdot)$ is a sequence $\mathcal{D} := (\mathcal{D}_n)_{n \in \mathbb{N}}$ such that for each $n \in \mathbb{N}$:
\begin{itemize}
\item [(a)] $\mathcal{D}_n$ is a Boolean algebra of subsets of $\mathbb{R}^n$.
\item [(b)] If $X \in \mathcal{D}_m$ and $Y \in \mathcal{D}_n$, then $X \times Y \in \mathcal {D}_{m+n}.$
\item [(c)] If $X \in \mathcal {D}_{n + 1},$ then $\pi(X) \in \mathcal {D}_n,$ where $\pi \colon \mathbb{R}^{n+1} \to \mathbb{R}^n$ is the projection on the first $n$ coordinates.
\item [(d)] $\mathcal{D}_n$ contains all algebraic subsets of $\mathbb{R}^n.$
\item [(e)] Each set belonging to $\mathcal{D}_1$ is a finite union of points and intervals.
\end{itemize}
}\end{definition}

A set belonging to $\mathcal{D}$ is said to be {\em definable} (in that structure). {\em Definable mappings} in structure $\mathcal{D}$ are mappings whose graphs are definable sets in $\mathcal{D}.$

Examples of o-minimal structures are
\begin{itemize}
\item the semi-algebraic sets (by the Tarski--Seidenberg theorem),
\item the globally sub-analytic sets, i.e., the sub-analytic sets of $\mathbb{R}^n$ whose (compact) closures in the real projective space $\mathbb{R}\mathbb{P}^n$ are sub-analytic (using Gabrielov's complement theorem).
\end{itemize}

In this note, we fix an o-minimal structure on $(\mathbb{R}, +, \cdot).$ The term ``definable'' means definable in this structure.
We recall some useful facts which we shall need later.

\begin{lemma} [{Monotonicity}] \label{MonotonicityLemma}
Let $f \colon (a, b) \rightarrow \mathbb{R}$ be a definable function and $p$ be a positive integer. Then there are finitely many points $a = t_0 < t_1 < \cdots < t_k = b$ such that the restriction of $f$ to each interval $(t_i, t_{i + 1})$ is of class $C^p$, and either constant or strictly monotone.
\end{lemma}

\begin{lemma}[Curve selection lemma at infinity]\label{selectioncurve}
Let $X \subset \R^n$ be a definable set, $f \colon X \to \R^m$ be a continuous definable mapping and $p$ be a positive integer. Assume that there exists a sequence $x_k \in X$ such that $\|x_k\|\to\infty$ and $f(x_k)\to y$ for some $y \in \R^m.$ Then there exists a definable $C^p$-curve $\gamma \colon [r, +\infty) \to \R^n,$ $t \mapsto \gamma (t),$ such that $\gamma (t)\in X,$ $\|\gamma (t)\| = t,$ and $\lim_{t\to \infty}f(\gamma (t))= y.$
\end{lemma}

\begin{lemma}\label{Lemma21}
Let $X \subset \mathbb{R}^n$ be an open definable set and let $f \colon X \to \mathbb R^m$ be a definable mapping. Then 
for each positive integer $p,$ the set of points where $f$ is not of class $C^p$ is a definable nowhere dense subset of $X$.
\end{lemma}

The following useful property of definable sets is the Koopman--Brown theorem (cf. \cite[Chapter 7, Theorem 4.2]{Dries1998}. For $x, x' \in \R^n$ we denote by $[x, x']$ the segment with endpoints $x$ and $x'$, i.e., $[x, x']:=\{(1-t)x+tx'  \colon 0\leq t\leq 1\}$.
\begin{lemma}[{Good directions lemma}]\label{Koopman-Brown}
Let $X \subset \mathbb{R}^n$ be a definable set. If $X$ is nowhere dense in $\R^n,$ then for each $x \in \R^n$ the set 
$$\{x' \in \R^n : [x, x'] \cap X \textrm{ is finite} \}$$ 
is dense and definable in $\R^n.$
\end{lemma}

By the cell decomposition theorem, for any $p \in \mathbb{N}$ and any nonempty definable subset $X$ of $\mathbb{R}^n,$ we can write $X$ as a disjoint union of finitely many definable $C^p$-manifolds of different dimensions. The {\em dimension} $\dim X$ of a nonempty definable  set $X$ can thus be defined as the dimension of the manifold of highest dimension of its decomposition. This dimension is well defined and independent of the decomposition of $X.$ By convention, the dimension of the empty set is taken to be negative infinity. 

\begin{lemma}\label{DimensionProposition1}
Let $X \subset \mathbb R^n$ be a definable set. Then 
\begin{enumerate}[{\rm (i)}]
\item If $X$ is nonempty then $\dim (\overline{X}\setminus X) < \dim X.$ In particular,  $\dim \overline{X} = \dim X.$
\item If $Y \subset \mathbb{R}^n$ is definable, then $\dim (X \cup Y) = \max \{\dim X, \dim Y\}.$
\item Let $f \colon X \rightarrow \mathbb R^m$ be a definable mapping. Then $\dim f(X) \leq \dim X.$
\end{enumerate}
\end{lemma}

\begin{lemma}\label{DenseIsOpen} 
Let $X \subset \mathbb R^n$ be a definable set. The following statements are equivalent.
\begin{enumerate}[{\rm (i)}]
\item $X$ is nowhere dense in $\mathbb{R}^n$;
\item $X$ has measure zero;
\item $\dim X < n.$
\end{enumerate}
\end{lemma}

\section{Regular points and values of continuous definable mappings} \label{Section3}

Let $\Omega$ be an open definable set in $\mathbb{R}^n,$ and let $f \colon \Omega \to \R^n$ be a continuous definable mapping. We denote by $B_f$ the set of points where $f$ is not of class $C^1$ (this is always a nowhere dense subset of $\Omega$ in view of Lemma~\ref{Lemma21}). In what follows, for any $x\in \Omega\setminus B_f$,  we denote by $df(x)$  the derivative of $f$ at $x$ and for any set $U \subset \Omega$,  by $df(U)$ we mean the set $\{df(x)\colon x\in U \setminus B_f\}.$ So the mapping 
$$df \colon \Omega \setminus B_f \to \mathbb{R}^{n\times n}, \quad x \mapsto df(x),$$ 
is well-defined. Observe that for each $x \in \Omega,$ we have $\B_\epsilon(x) \subset \Omega$ for some $\epsilon > 0$ and the function
$$(0, \epsilon) \to \mathbb{R}, \quad r \mapsto \inf_{A\in {\rm co } (df(\B_r(x)))}\sigma (A),$$
is decreasing. (Recall that $\sigma(A)$ denotes the co-norm of $A \in \mathbb{R}^{n \times n}.$) Hence, the value
$$\nu_f (x):=\lim_{r\to 0^+} \inf_{A\in {\rm co } (df(\B_r(x)))}\sigma (A)$$
is well-defined and nonnegative (possibly infinite). By definition, it is clear that $\nu_f(x) = \sigma(df(x))$ for all $x \in \Omega \setminus B_f.$
Furthermore, it is not hard to check that the function $\Omega\to\R\cup \{+\infty\}, x \mapsto \nu_f(x),$ is lower semicontinuous; as we shall not use this fact, we leave the proof as an exercise.

\begin{definition} 
A point $x\in \Omega$ is a {\em regular point} of $f$ if $\nu _f(x) > 0$, and $x$ is a {\em critical point} of $f$ if $\nu _f(x) = 0.$ 
An element $y \in \mathbb{R}^n$ is a {\em regular value} of $f$ if either $f^{-1}(y)$ is an empty set or all the points in $f^{-1}(y)$ are regular points, and  $y$ is a {\em critical value} of $f$ if the set $f^{-1}(y)$ contains at least a critical point of $f.$
\end{definition}

\begin{remark} 
(i) By definition, it is not hard to see that if the mapping $f$ is of class $C^1$ the above concepts coincide with the classical ones.

(ii) As the reader can see, the above definition of the function $\nu_f(\cdot)$ follows Clarke's idea of taking the convex hull of all Jacobian matrices. To see this, let $x \in \Omega$ be a point near which $f$ is Lipschitz. Clarke \cite{Clarke1990} defined the {\em generalized Jacobian} $Jf(x)$ as the convex hull of all matrices which are limits of Jacobian matrices $df(x')$ as $x' \to x$ with $f$ being differentiable at $x'.$ By definition, $x$ is a regular point of $f$ (i.e., $\nu_f(x) > 0$) if and only if every matrix in $Jf(x)$ is non-singular.
\end{remark}

\begin{example} \label{Example2}  
(i) Consider the continuous semi-algebraic mapping 
$$f \colon \R^2\to\R^2, \quad x:= (x_1, x_2) \mapsto f(x) := (\sqrt[3]{x_1},x_2).$$
For any $x \in \mathbb{R}^2$ with $x_1\neq 0$, we have
		$df(x)= 
		\begin{bmatrix}
		\frac{1}{3\sqrt[3]{x_1^2}} & 0 \\
		0 & 1 
		\end{bmatrix}.$
		One can check that $\nu_f(0, 0) = 1$ and thus, $(0, 0)$ is a regular point of $f$.  Since $f^{-1}(0,0)=\{(0,0)\}$, it follows that  $y=(0,0)$  is a regular value of $f$.
		
(ii) Consider the continuous semi-algebraic mapping 
$$f \colon \R^2\to\R^2, \quad x:= (x_1, x_2) \mapsto f(x) :=(\sqrt[3]{x_1}, \sqrt[3]{x_2}).$$
For any $x \in \mathbb{R}^2$ with $x_1\neq 0$ and $x_2\neq 0$, we have
		$df(x)= 
		\begin{bmatrix}
		\frac{1}{3\sqrt[3]{x_1^2}} & 0 \\
		0 & \frac{1}{3\sqrt[3]{x_2^2}} 
		\end{bmatrix}.$
				One can check that $\nu_f(0, 0) = +\infty$ and thus,  $(0, 0)$ is a regular point of $f.$    Since $f^{-1}(0,0)=\{(0,0)\}$, it follows that  $y=(0,0)$  is a regular value of $f$.
		
(iii) Consider the continuous semi-algebraic mapping 
$$f \colon \R^2\to\R^2, \quad x := (x_1, x_2) \mapsto f(x) := (\sqrt{|x_1|},x_2).$$ 
We have
		$df(x_1,x_2)=
		\begin{bmatrix}
		\frac{1}{2\sqrt{x_1}} & 0 \\
		0 & 1 
		\end{bmatrix}$ if 
		$x_1 > 0$ and 
	$	df(x_1,x_2)=
	\begin{bmatrix}	-\frac{1}{2\sqrt{-x_1}} & 0 \\
		0 & 1 
		\end{bmatrix}$ if  $x_1<0$.
Observe that for any $r>0$ and $(x_1,x_2) \in \B_r(0)$ with $x_1 \ne 0,$ the matrix 
$$\frac{1}{2} \big(df(x_1,x_2) + df(-x_1, x_2) \big)$$ 
is not invertible. Therefore, $\nu_f(0,0)=0$ and so $x=(0,0)$ is a critical point of $f$.
\end{example}

Although a given value $y$ need not be a regular value of $f,$ the following claim ensures that almost all values are regular; see also \cite{Ioffe2007} for related results.

\begin{theorem}[Sard theorem]
Let $\Omega\subset \R^n$ be an open definable set and $f \colon  \Omega\to \R^n$ be a continuous definable mapping. Then the set of critical values of $f$ is of dimension strictly less than $n,$ and so, it has measure zero.
\end{theorem}
\begin{proof}
In view of Lemma~\ref{Lemma21}, $B_f$ is a definable nowhere dense subset of $\Omega.$ This, together with Lemmas~\ref{DimensionProposition1}~and~\ref{DenseIsOpen}, implies that 
$$\dim f(\overline{B_f}) \le \dim \overline{B_f} = \dim B_f < \dim \Omega = n.$$

On the other hand, since the restriction of $f$ on $\Omega \setminus \overline{B_f}$ is of class $C^1,$ the classical Sard theorem ensures that the set $f(\Sigma_f)$ has measure zero in $\mathbb{R}^n,$ where $\Sigma_f := \{x \in \Omega \setminus \overline{B_f} : \det df(x) = 0 \}.$ Since $f$ is definable, $\Sigma_f,$ and hence $f(\Sigma_f),$ are definable. According to Lemma~\ref{DenseIsOpen}, then $\dim f(\Sigma_f) < n.$  Therefore, by Lemma~\ref{DimensionProposition1}, 
$$\dim ( f(\overline{B_f}) \cup f(\Sigma_f))= \max \{\dim  f(\overline{B_f}), \dim f(\Sigma_f)\}<n.$$
Finally, since  the set of critical values of $f$ is a subset of $f(\overline{B_f}) \cup f(\Sigma_f)$,  the desired conclusion follows.
\end{proof}

Let us establish some properties of regular points that will be used in the calculation of degree for continuous definable mappings.

\begin{lemma} \label{Claim31}
$\overline{x} \in \Omega$ is a regular point of $f$ if and only if there exist scalars $r > 0$ and   $\delta > 0$ such that the following inequalities hold
\begin{eqnarray*}
\sigma(A) &\ge& \delta \quad \textrm{ for all } \quad A \in \mathrm{co} (df(\B_r(\overline{x}))), \\
\nu_f(x) &\ge& \delta \quad \textrm{ for all } \quad x \in \B_r(\overline{x}).
\end{eqnarray*}
In particular, the set of regular points of $f$ is an open subset of $\Omega.$
\end{lemma}

\begin{proof} 
It suffices to prove the ``only if" part. 
We only consider the case $\nu_f(\overline{x}) < +\infty.$ The case where $\nu_f(\overline{x}) = +\infty$ is proved similarly.
	
By definition, there exist $r > 0$ and $\delta > 0$ such that
$$\inf_{A \in \mathrm{co} (df(\B_r(\overline{x}))) }\sigma(A)  \ge \frac{1}{2}\nu_f(\overline{x}).$$
Letting $\delta := \frac{1}{2}\nu_f(\overline{x}) > 0,$ we get the first inequality.
	
Take any $x \in \B_r(\overline{x}).$ For all $\rho > 0$ sufficiently small, we have $\B_{\rho} (x) \subset \B_r(\overline{x}),$ and so
$$\inf_{A \in \mathrm{co} (df( \B_{\rho} (x))) }\sigma(A)  \ge \inf_{A \in \mathrm{co} (df(\B_r(\overline{x}))) }\sigma(A) \ge \delta.$$
Letting $\rho \to 0^+,$ we get the second inequality.
\end{proof}

\begin{lemma}\label{Claim32}
If $\overline{x} \in \Omega$ is a regular point of $f,$ then there exists a scalar $r > 0$ such that $\det A$ is negative (or positive) for all $A \in \mathrm{co} (df(\B_r(\overline{x}))).$
\end{lemma}

\begin{proof} 
By Lemma~\ref{Claim31}, there exist scalars $r > 0$ and  $\delta > 0$ such that for each $A \in \mathrm{co} (df(\B_r(\overline{x}))),$ we have $\sigma(A) \ge \delta$ and so $\det A \ne 0.$ We show that the function 
$$\mathbb{R}^{n \times n} \to \mathbb{R}, \quad A \mapsto \det A,$$ 
has a constant sign on the set $\mathrm{co} (df(\B_r(\overline{x}))).$ By contradiction, suppose that there exist $A_1$, $A_2 \in \mathrm{co}(df(\B_r(\overline{x})))$ such that 
$$\det A_1 < 0 < \det A_2.$$ 
Consider the function $g \colon [0, 1] \to \R$ given by 
$$g(t)={\rm det}((1 - t)A_1 + tA_2).$$
Clearly, $g$ is continuous and $g(0) < 0 < g(1).$ Hence we must have $g(t) = 0$ for some $t \in (0, 1).$ This is a contradiction, because we know that $(1 - t)A_1 + tA_2 \in \mathrm{co} (df(\B_r(\overline{x})))$ and ${\rm det}((1 - t)A_1 + tA_2) \neq 0.$
\end{proof}

\section{The mean value theorem, the inverse function theorem and the implicit function theorem for continuous definable mappings} \label{Section4}

In this section we present some classical analysis results for continuous definable mappings. The first one is as follows (see also \cite{Lebourg1975, Pourciau1977}).

\begin{theorem}[Mean value theorem] \label{MeanValueTheorem}  
Assume that  $\Omega\subseteq \R^n$ is an open definable set and that $f \colon  \Omega\to \R^n$ is  a continuous definable mapping.
Let $x \in \Omega.$ There exists a dense and definable subset $\Omega(x)$ of $\Omega$ such that for each $x' \in \Omega(x)$ with $[x, x'] \subset \Omega$ we have
$$f(x') - f(x) \in \mathrm{\overline{co}} \{d f [x, x'] (x' - x)\},$$
where   $d f [x, x'] (x' - x) := \{ A(x' - x) : A \in df([x, x']) \}.$
\end{theorem}

\begin{proof}
Since $B_f$ is a nowhere dense subset of $\Omega,$ it follows from Lemma~\ref{Koopman-Brown} that the set 
$$\Omega(x) := \{x' \in \R^n : [x, x'] \cap B_f \textrm{ is finite} \}$$ 
is dense and definable in $\Omega.$ Assume that the desired conclusion does not hold at some point $x' \in \Omega(x)$ with $[x, x'] \subset \Omega.$ Since $\mathrm{\overline{co}} \{ d f [x, x'] (x' - x)\}$ is a closed  convex subset of $\mathbb{R}^n,$ it follows from the separation theorem that
$$\langle u, f(x') - f(x) \rangle - \epsilon > \langle u, v \rangle \quad \textrm{ for all } \quad v  \in \mathrm{\overline{co}} \{ d f [x, x'] (x' - x)\}$$
for some $u \in \mathbb{R}^n$ and $\epsilon > 0.$ Consider the function 
$$g \colon [0, 1] \to \mathbb{R}, \quad t \mapsto g(t) := \langle u, f((1 - t)x + tx') \rangle.$$
For almost all (in fact, all but finite) $t \in [0, 1],$ the function $f$ is differentiable at $(1 - t)x + tx'.$ For such $t$, we have
$$g'(t) = \langle u, df((1 - t)x + tx')(x' - x) \rangle < \langle u, f(x') - f(x) \rangle - \epsilon.$$
Hence
$$g(1) - g(0) = \int_0^1 g'(t) dt < \langle u, f(x') - f(x) \rangle - \epsilon,$$
which is impossible because $g(1) - g(0) = \langle u, f(x') - f(x) \rangle.$
\end{proof}

\begin{remark}
(i) We cannot drop $\mathrm{\overline{co}}$ in the conclusion of Theorem~\ref{MeanValueTheorem}. Indeed, consider the continuous semi-algebraic function $f \colon  \R\to\R, x \mapsto -|x|.$ A simple calculation shows that
$$f(-1)-f(1)=0\notin d f [-1, 1] (-1-1)=\{-2,2\}.$$

(ii) In Theorem~\ref{MeanValueTheorem}, the set $\Omega(x)$ may not equal to $\Omega.$ Indeed, consider the continuous semi-algebraic function $f \colon \Omega \to \R^2, (x_1, x_2) \mapsto (x_1, \sqrt[3]{x_2}),$ where $\Omega := \R^2.$ Clearly, $B_f = \{(x_1, x_2) \in \mathbb{R}^2 : x_2 = 0\}.$ Moreover, for $x := (0, 0)$ we have $\Omega (x) \subseteq (\R^2 \setminus B_f) \cup \{(0,0)\},$ and so $\Omega(x) \ne \Omega.$
\end{remark}

The following fact plays a crucial role in the proof of Theorems~\ref{InverseTheorem} and \ref{ImplicitTheorem} below.

\begin{corollary}\label{MeanValueTheorem2}
Assume that  $\Omega\subseteq \R^n$ is an open definable set and that $f \colon  \Omega\to \R^n$ is  a continuous definable mapping.
Let $\overline{x} \in \Omega$ and $r > 0$ be such that $\B_r(\overline{x}) \subset \Omega.$ Then 
\begin{eqnarray*}
\|f(x') - f(x) \| & \ge & \inf_{A \in \mathrm{co} (df(\B_r(\overline{x})))} \| A(x' - x)\|
\end{eqnarray*}
for all $x, x' \in \B_r(\overline{x}).$
\end{corollary}

\begin{proof}
Let $x \in \B_r(\overline{x}).$ By Theorem~\ref{MeanValueTheorem}, there exists a dense subset $\Omega(x)$ of $\Omega$ such that for each $x' \in \Omega(x)$ with $[x, x'] \subset \Omega$ we have
$$f(x') - f(x) \in \mathrm{\overline{co}} \{ d f [x, x'] (x' - x)\}.$$
Take any $x' \in \B_r(\overline{x}) \cap \Omega(x).$ Then $[x, x'] \subset \B_r(\overline{x}) \subset \Omega.$ Hence
\begin{eqnarray*}
f(x') - f(x) 
& \in & \mathrm{\overline{co}} \{ d f [x, x'] (x' - x)\} \\
& \subset & \mathrm{\overline{co}} \{ df (\B_r(\overline{x})) (x' - x)\}.
\end{eqnarray*}
Consequently,
\begin{eqnarray*}
\|f(x') - f(x) \|
& \ge & \inf_{A \in \mathrm{co} (df(\B_r(\overline{x})))} \| A(x' - x)\|.
\end{eqnarray*}
Clearly, this inequality still holds for all $x' \in \B_r(\overline{x})$ because of the density of $\Omega(x)$ in $\Omega.$ 
\end{proof}

For further reference we recall here the classical invariance of domain theorem.

\begin{lemma} [Invariance of domain] \label{InvarianceDomain}
Let $\Omega$ be an open subset of $\mathbb{R}^n.$  Then every injective continuous mapping $f \colon \Omega \to \R^n$ is open, i.e. the image of an open subset of $\Omega$ under $f$ is an open set in $\mathbb{R}^n.$
 \end{lemma}
 \begin{proof}
 See, for example, \cite[Chapter~4, Proposition~7.4]{Dold1972}.
 \end{proof}
 
The following result is inspired by the work of Fukui, Kurdyka and P\u{a}unescu \cite{Fukui2009}.

\begin{theorem}[Inverse function theorem]  \label{InverseTheorem}  
Assume that  $\Omega\subseteq \R^n$ is an open definable set and that $f \colon  \Omega\to \R^n$ is  a continuous definable mapping.
 Let $\overline{x} \in \Omega$ be a regular point of $f.$ Then there exist scalars $r>0$ and   $\delta > 0$ such that
$$\|f(x') - f(x)\|\geq \delta \|x'-x\| \quad {\rm for\ all} \quad  x, x'\in \B_r(\overline{x}).$$
In particular, $f(\B_r(\overline{x}))$ is an open set in $\mathbb{R}^n$ and the restriction mapping $f \colon \B_r(\overline{x})\to f(\B_r(\overline{x}))$ is a homeomorphism with Lipschitz definable inverse.
\end{theorem}

\begin{proof}
In light of Lemma~\ref{Claim31}, there exist scalars $r > 0$ with $\B_r(\overline{x}) \subset \Omega$ and $\delta > 0$ such that 
\begin{eqnarray*}
\sigma(A) &\ge& \delta \quad \textrm{ for all } \quad A \in \mathrm{co} (df(\B_r(\overline{x}))).
\end{eqnarray*}
Take any $x, x' \in \B_r(\overline{x}).$ It follows from Corollary~\ref{MeanValueTheorem2} that
\begin{eqnarray*}
\|f(x') - f(x) \|
& \ge & \inf_{A \in \mathrm{co} (df(\B_r(\overline{x})))} \| A(x' - x)\| \\
& \ge & \inf_{A \in \mathrm{co} (df(\B_r(\overline{x})))} \sigma(A) \|x' - x\|,
\end{eqnarray*}
and so
\begin{eqnarray*}
\|f(x') - f(x) \| & \ge & \delta \|x' - x\|.
\end{eqnarray*}
Therefore, $f$ is injective on $ \B_r(\overline{x}).$ 
By Lemma~\ref{InvarianceDomain}, the restriction mapping 
 $$f \colon \B_r(\overline{x}) \to f(\B_r(\overline{x}))$$
 is open, so indeed, it is a homeomorphism with Lipschitz  inverse. Finally, the definability of the inverse mapping follows directly from the definition.
\end{proof}

The following claim is a non-smooth version of the  implicit function theorem.
	\begin{theorem}[Implicit function theorem]  \label{ImplicitTheorem}
Let $\Omega \subseteq \R^n\times \R^p$ be an open definable set and $F \colon \Omega\to \R^p,  (x, y) \mapsto F(x,y),$ be a continuous definable mapping. Assume that $ (\overline{x}, \overline{y})$ is a point in $\Omega$ such that $F(\overline{x}, \overline{y}) = 0$ and
\begin{equation}\label{Condition}
\lim_{r\to 0^+}\inf _{ A\in \mathrm{co} (dF_x(\B_r(\overline{y}))), \, x\in \B_r( \overline{x})}\sigma (A)>0,
\end{equation}
where for each $x$ (near $\overline{x}$), $F_x$ stands for the mapping $y\mapsto F(x,y)$.
Then there are open definable neighborhoods $U$ of $\overline{x}$ in  $\mathbb{R}^n$ and $V$ of $\overline{y}$ in  $\mathbb{R}^p$ and a unique continuous definable mapping $G \colon U \to V$ such that $G(\overline{x}) = \overline{y}$ and 
$$F(x, G(x)) = 0 \quad \textrm{ for all } \quad x \in U.$$
\end{theorem}

\begin{proof} 
Consider the continuous definable mapping
$$f \colon \Omega \to \R^{n} \times \R^{p}, \quad (x, y) \mapsto (x, F(x, y)).$$
By the inequality~\eqref{Condition}, there exists a scalar  $r>0$ such that 
$$\inf _{ A\in \mathrm{co}(dF_x(\B_r(\overline{y}))), \, x\in \B_r(\overline x)}\sigma (A)>0.$$
We first prove that the restriction of $f$ on $\B_r(\overline{x}, \overline{y})$ is injective. To see this, let $(x, y), (x', y') \in \B_r(\overline{x}, \overline{y})$ be such that $f(x,y) = f(x,y').$ Clearly, $x = x'.$ 
Then $$(0,0)=f(x,y)-f(x',y')=(0, F(x,y)-F(x,y'))=(0,F_x(y)-F_x(y')).$$
By Corollary~\ref{MeanValueTheorem2},
we have
\begin{eqnarray*}
	0=\|F_x(y)-F_x(y')\|
	& \ge & \inf_{A \in \mathrm{co} (dF_x(\B_r(\overline{y})))} \|A(y-y')\| \\
	& \ge & \inf_{A \in \mathrm{co} (dF_x(\B_r(\overline{y})))}\sigma (A) \|y-y'\|,
\end{eqnarray*}
which yields $y = y'.$ Hence the restriction mapping
$$f \colon \B_r(\overline{x}, \overline{y}) \to f(\B_r(\overline{x}, \overline{y}))$$
is injective.

 By Lemma~\ref{InvarianceDomain}, the mapping $f$  is open, so indeed, it is a homeomorphism with the continuous definable inverse $f^{-1}$.  
More precisely, we can find open  neighborhoods $W_1\subset \R^{n+p}$ and $W_2\subset \R^{n+p}$  of $(\overline{x},\overline{y})$ and $f(\overline{x},\overline{y})=(\overline{x},0)$, respectively,  so that $f \colon W_1\to W_2$ and  $f^{-1} \colon W_2\to W_1$ are  bijective.
Since $f(x,y)=(x,F(x,y)),$  we can write 
$$f^{-1}(x,u)=(x,g(x,u)),$$
 where  $g\colon W_2\to \R^p$ is a  continuous definable mapping.  Observe that $g(\overline{x},0)=\overline{y}$.
Denote
$$U:= \{x\in \R^n \colon (x,0)\in W_2\}$$
and
$$V:= \{y\in \R^p \colon \exists x\in \R^n, \  (x,y)\in W_1\}.$$
One can check that $U$ and $V$ are  definable open neighborhoods of $\overline{x}$  and $\overline{y}$, respectively. Let $x\in U$. Then $(x,0)\in W_2$ and  $f^{-1}(x,0)=(x,g(x,0))\in W_1$. Hence $g(x,0)\in V$.  Let  $G$ be the continuous definable mapping defined by
$$G\colon U\to V, \quad x\mapsto g(x,0).$$
Then $G(\overline{x}) = g(\overline{x}, 0) = \overline{y}.$ Next, take any $x\in U$. By definition, 
 $(x,0)\in W_2$. It follows that
 $$(x,0)=f(f^{-1}(x,0))=f(x,g(x,0))=(x, F(x,g(x,0)))=(x,F(x,G(x)))$$
and so $F(x,G(x))=0$. 
 
Finally, we prove the uniqueness of the mapping $G$. For this, suppose that   $G_0\colon U\to V$ is some continuous  definable mapping satisfying $F(x, G_0(x))\equiv 0$ on $U$. Take any $x\in U$. We show that $G(x)=G_0(x)$. Indeed, since $(x,0)\in W_2$ and $F(x,G_0(x))=0$, we have
$$f(x,G_0(x))=(x,F(x,G_0(x)))=(x,0).$$
The mapping $f^{-1}$ maps the point $f(x,G_0(x))\in W_2$ to the point $(x,G_0(x))\in W_1$ and, moreover, 
$$(x,G_0(x))=f^{-1}(f (x, G_0(x)))=f^{-1}(x,0)=(x,g(x,0))=(x,G(x)).$$
Hence $G_0(x)=G(x)$. 
\end{proof}

\begin{remark} 
By definition, it is easy to check that if $F$ is of class $C^1$, the inequality~\eqref{Condition} is equivalent to the fact that the derivative mapping $dF_{\overline{x}}(\overline {y}) \colon \mathbb{R}^p \to \mathbb{R}^p$ is surjective,
which is required in the classical implicit function theorem.
\end{remark}
	
\section{A formula for the Brouwer degree of continuous definable mappings} \label{Section5}

Suppose $\Omega$ is an open and bounded set in $\mathbb{R}^n,$ $f \colon \overline{\Omega} \rightarrow \mathbb{R}^n$ is a continuous mapping, and $y \not \in f(\partial \Omega).$ The {\em Brouwer degree} of $f$ on $\Omega$ with respect to $y,$ denoted $\deg(f, \Omega, y),$ is a function with values in $\Z$ which enjoys several important properties (normalization, domain decomposition, local constancy, homotopy invariance, etc.). The Brouwer degree is a power tool used in analysis and topology, in particular, it gives an estimation and the nature of the solution(s) of the equation $f(x) = y$ in $\Omega$ (for example, when this degree is nonzero, the equation $f(x) = y$ has a solution in $\Omega).$

The two propositions below provide some useful properties of the Brouwer degree; for more details, please refer to \cite{Deimling1985, LLoyd1978} and the references therein. 

\newpage
\begin{proposition}\label{Proposition51}
The following statements hold.
\begin{enumerate}[{\rm (i)}]
\item {\rm (Domain decomposition)} If $\Omega_1$ and $\Omega_2$  are disjoint open subsets of $\Omega$ such that $y\notin f(\overline{\Omega} \setminus (\Omega_1\cup \Omega_2))$, then $$\deg(f,\Omega, y) = \deg(f,\Omega_1,y) + \deg(f,\Omega_2,y).$$
\item {\rm (Local constancy)} $\deg(f, \Omega, \cdot)$ is constant on $\B_r(y),$ where $r := \mathrm{dist} (y, f(\partial \Omega))$-the distance from $y$ to $f(\partial \Omega),$ that is
$$\deg(f, \Omega, y') = \deg(f, \Omega, y) \quad \textrm{ for all } \quad y' \in \B_r(y).$$

\item {\rm (Homotopy invariance)} If $H \colon \overline{\Omega} \times [0,1]\to\R^n$ is continuous, $y \colon  [0,1]\to\R^n$ is continuous, and $y(t) \not\in H(\partial \Omega, t)$ for every $t\in [0,1]$, then $\deg(H(\cdot, t), \Omega, y(t))$ is independent of $t\in [0,1]$.
\end{enumerate}
\end{proposition}

\begin{proposition}\label{Proposition52}
Let $f$ be of class $C^1$ on $\Omega,$ and assume that $y \not \in f(\partial \Omega)$ be a regular value of $f.$ Then
$$\deg (f, \Omega, y) = \sum_{x \in f^{-1}(y)} \mathrm{sign} \det df(x).$$
\end{proposition}

The sum above is finite, since $f^{-1}(y)$ is a discrete set by the inverse function theorem and $\overline{\Omega}$ is compact.
As shown in Theorem~\ref{DegreeTheorem} below, in our setting, this representation is still true even $f$ is {\em not} differentiable at every $x \in f^{-1}(y).$ 

We now suppose that $f \colon \overline{\Omega} \to \mathbb{R}^n$ is a continuous definable mapping. Recall that the set $B_f$ of the points where $f$ is not of class $C^1$ is a nowhere dense definable subset of $\Omega.$ Thanks to Lemma~\ref{Claim32}, the following definition is well-defined.

\begin{definition} 
For a regular point  $\overline{x} \in \Omega$  of $f$, we define ${\rm Sign}_f(\overline{x})$ as the number given by
$${\rm Sign}_f(\overline{x}) := \lim_{x \in \Omega \setminus B_f, x \to \overline{x}} {\rm sign \det} df(x).$$
\end{definition}

The next lemma shows that ${\rm Sign}_f(\cdot)$ is locally constant around  every regular point $\overline{x}$ of $f.$ Recall that by  Lemma~\ref{Claim31}, there exists an open neighborhood of $\overline{x}$ such that all points in this neighborhood are regular.

\begin{lemma}\label{Claim33} 
	If $\overline{x} \in \Omega$ is a regular point of $f$,  then there exists a scalar $r > 0$ such that 
	$${\rm Sign}_f(x) = {\rm Sign}_f(\overline{x})	\quad \textrm{ for  all } \quad x\in \B_{r}(\overline{x}).$$
\end{lemma}

\begin{proof}  
In view of Lemma~\ref{Claim32}, there exists a scalar $r > 0$ such that the function 
$$\mathbb{R}^{n \times n} \to \mathbb{R}, \quad A \mapsto \det A,$$ 
has a constant sign on $\mathrm{co} (df(\B_r(\overline{x}))).$ By definition, the desired result follows at once. 
\end{proof}

The main claim of this section, which is inspired by the works of Pourciau \cite{Pourciau1983-2} and Shannon \cite{Shannon1994},  is as follows.

\begin{theorem}\label{DegreeTheorem}
Let $\Omega \subset \mathbb{R}^n$ be an open, bounded and definable set and $f \colon \overline{\Omega} \to \mathbb{R}^n$ be a continuous definable mapping. The following statements hold.
\begin{enumerate}[{\rm (i)}]
\item Suppose that $x \in \Omega$ is a regular point of $f,$ and $y := f(x).$ Let $W \subset \Omega$ be an open neighborhood of ${x}$ such that $\overline{W} \cap f^{-1}(y) = \{x\}.$ Then 
$$\deg(f,W,y) = \mathrm{Sign}_f(x).$$

\item Let $y \not \in f(\partial \Omega)$ be a regular value of $f.$ Then
$$\deg(f,\Omega, y)=\sum_{x\in f^{-1}(y)} \mathrm{Sign}_f(x).$$
\end{enumerate}
\end{theorem}

\begin{proof}
(i) By Lemmas~\ref{Claim31}, \ref{Claim33}, and Theorem~\ref{InverseTheorem}, we can find a scalar $r > 0$ with $\B_r(x)\subset W$ such that the following properties hold true:
\begin{enumerate}[{\rm $\bullet$}]
\item every $x' \in \B_r(x)$ is a regular point of $f$ and satisfies $\mathrm{Sign}_f(x') = \mathrm{Sign}_f(x);$ and
\item the restriction mapping $f \colon \B_r(x) \to f(\B_r(x))$ is a homeomorphism.
\end{enumerate}
We know that (see Lemmas~\ref{Lemma21}, \ref{DimensionProposition1} and \ref{DenseIsOpen})
$$\dim \overline{B_f} = \dim B_f < n.$$
Consequently, we can find a point $x' \in \B_r(x)\setminus \overline {B_f}$ such that $y' := f(x') \in \mathbb{B}_r(y)$ with $r := \mathrm{dist}(y, f(\partial{\B_r(x)})).$
By Proposition~\ref{Proposition51}(ii), then
 $$\deg(f, \B_r(x), y) = \deg(f, \B_r(x), y').$$
Clearly, there exists an open set $V$ in $\R^n$ with $x'\in V$ and $\overline{V}\subset \B_r(x)\setminus \overline {B_f}$. In light of Proposition~\ref{Proposition51}(i), we have
$$\deg(f, \B_r(x), y') = \deg(f, V, y').$$
On the other hand, since the restriction of $f$ on $V$ is of class $C^1,$ it follows from Proposition~\ref{Proposition52} that
$$\deg(f, V, y') = \mathrm{sign} \det df(x').$$
Therefore
\begin{eqnarray*}
\deg(f, W, y) &=& \deg(f, \B_r(x), y) \ = \ \deg(f, \B_r(x), y')\\
&=& \deg(f, V, y') \ = \ \mathrm{sign} \det df(x')\\
&=& \mathrm{Sign}_f(x),
\end{eqnarray*}
where the first equality follows directly from Proposition~\ref{Proposition51}(i).

(ii) Now assume that $y$ is a regular value of $f$ and $y \not \in f(\partial \Omega).$ It follows from the inverse function theorem (Theorem~\ref{InverseTheorem}) that $f^{-1}(y)$ is a finite set $\{x_1, \ldots, x_m\}$ and there exist open neighborhoods $W_k \subset \Omega$ of $x_k$ for each $k = 1, \ldots, m,$ such that $\overline{W_k} \cap f^{-1}(y) = \{x_k\}.$ Since $y \not \in f(\overline{\Omega} \setminus \cup_{k = 1}^m W_k),$ the domain decomposition property of the Brouwer degree (Proposition~\ref{Proposition51}(i)) guarantees that 
$$\deg(f, \Omega, y) = \sum_{k = 1}^m \deg(f, W_k, y).$$
This, together with the first claim, establishes the second claim.
\end{proof}

\begin{example}
Consider the continuous semi-algebraic mapping given in Example~\ref{Example2}(ii):
$$f \colon \R^2\to\R^2, \quad x:= (x_1, x_2) \mapsto f(x) :=(\sqrt[3]{x_1}, \sqrt[3]{x_2}).$$
We know that $y := (0, 0)$ is a regular value of $f$ and $f^{-1}(y) = \{(0, 0)\}.$ 
Since $f$ is not locally Lipschitz and is not differentiable at $\overline{x} := (0, 0),$ \cite[Theorem~3.5]{Pourciau1983-2} and \cite[Theorem~10]{Shannon1994} cannot be applied. On the other hand, a simple calculation shows that $\mathrm{Sign}_f(\overline{x}) = 1,$ and so $\deg(f, \mathbb{B}_r(\overline{x}), y) = 1$ for every $r > 0$ in view of Theorem~\ref{DegreeTheorem}.
\end{example}

\section{Global properties of continuous definable mappings} \label{Section6}

In this section we study the surjectivity, global invertibility and openness of continuous definable mappings. 
Recall that a continuous mapping $f \colon \mathbb{R}^n \to \mathbb{R}^n$ is said to be {\em  proper} if the preimage of a compact set is compact (this is equivalent to requiring that $\|f(x)\| \to \infty$ as $\|x\| \to \infty$) and $f$ is {\em quasi-interior} (see \cite{Whyburn1942}) if $y$ lies in the interior of $f(U)$ for every $y$ and every open set $U$ that contains a compact component of $f^{-1}(y).$ 

The next result is inspired by the work of Pourciau~\cite[Theorem~C]{Pourciau1983-1}. To see related results for smooth mappings, consult \cite{Church1963, Church1967, Denkowski2017, Gamboa1996, Hirsch2002, Sternberg1959}.

\begin{theorem}\label{PourciauTheorem}
Let $f \colon \mathbb{R}^n \to \mathbb{R}^n$ be a proper, continuous and definable mapping. Assume the following two conditions hold:
\begin{enumerate}[{\rm (i)}]
\item $\det df(x)$ is nonnegative (or nonpositive) for all $x \in \mathbb{R}^n \setminus B_f$; and 
\item the set of critical points of $f$ has measure zero. 
\end{enumerate}
Then $f$ is surjective and quasi-interior.
\end{theorem}
\begin{proof}
We break the proof in two parts.
	
\subsubsection*{The mapping $f$ is surjective}
Observe that  if there exists an $\overline{y} \in \mathbb{R}^n$ with the property that $\deg(f, V, \overline{y}) \ne 0$ whenever $V$ is an open ball centered at the origin and containing $f^{-1}(\overline{y}),$ then $f$ is surjective. To see this, choose any $y$ in $\mathbb{R}^n$ and let $\gamma$ stand for the line segment connecting $\overline{y}$ to $y,$  and $\gamma(t) = (1 - t)\overline{y} + t y$ for $0 \le t \le 1.$  As $f$ is proper and $\gamma$ is compact, $f^{-1}(\gamma)$ is compact. We can therefore find an open ball $V := \{x \in \mathbb{R}^n \ : \ \| x \| < \delta\}$ containing $f^{-1}(\gamma).$ By the homotopy invariance of degree (see Proposition~\ref{Proposition51}(iii)), then 
$$\deg(f, V, y) = \deg(f, V, \gamma(t)) = \deg(f, V, \overline{y}) \ne 0.$$ 
Hence, $f(x) = y$ for some $x$ in $V,$ and so $f$ is surjective.
	
To prove that $f$ is surjective, we need to  find $\bar{y}$ with the mentioned property. Recall that the set $B_f$ of the points where $f$ is not of class $C^1$ is a nowhere dense definable subset of $\mathbb{R}^n.$ Consider the definable set
\begin{eqnarray*}
S &:=& \{x \in \mathbb{R}^n \setminus B_f \ : \ \nu_f(x) = 0 \} \\
 &\,=& \{x \in \mathbb{R}^n \setminus B_f \ : \ \det df (x) = 0 \}.
\end{eqnarray*}
Then the set $S$ has measure zero because of the second condition (ii). Hence 
$\dim (B_f \cup S) < n$ and $\det df(\overline{x}) \ne 0$ for some $\overline{x} \not \in B_f \cup S.$ By the inverse mapping theorem (Theorem~\ref{InverseTheorem}), the image $f(\mathbb{R}^n)$ must contain an open set $W.$ On the other hand, it follows from Lemma~\ref{DimensionProposition1} that
$$\dim f(B_f \cup S) \le \dim (B_f \cup S)  < n.$$
So there is some $\overline{y}$ in $W$ that is not in $ f(B_f \cup S) .$ Then $f^{-1}(\overline{y})$ is nonempty, and it is compact because $f$ is proper. Suppose $V$ is any open ball centered at the origin and containing $f^{-1}(\overline{y}).$ As $\overline{y}$ is not in $f (\partial V),$ the degree $\deg (f, V, \overline{y})$ is defined. Note that for all $x \in f^{-1}(\overline{y})$ we have $x \in \mathbb{R}^n \setminus (B_f \cup S)$ and so $\det df(x) \ne 0.$ In particular, $\overline{y}$ is a regular value of $f.$ In light of Theorem~\ref{DegreeTheorem}, we have
\begin{eqnarray*}
\deg (f, V, \overline{y}) &=& \sum_{x \in V \cap f^{-1}(\overline{y})} \mathrm{sign} \det df(x).
\end{eqnarray*}
Now the condition (i) implies that
\begin{eqnarray*}
\sum_{x \in f^{-1}(\overline{y})} \mathrm{sign} \det df(x) &\ne& 0.
\end{eqnarray*}
Therefore $\deg (f, V,\overline{y}) \ne 0,$ and this completes the proof that $f$ is surjective.
	
\subsubsection*{The mapping $f$ is quasi-interior} Choose any $y \in f(\mathbb{R}^n) = \mathbb{R}^n$ and suppose $C$ is a connected component of $f^{-1}(y).$ Since $f$ is proper, $C$ is compact. Let $U$ be an open set containing $C.$ Then we can find a bounded, open and definable set $V$ containing $C$ with $\overline{V} \subset U$ and $f^{-1}(y) \cap \partial V = \emptyset.$ So the degree $\deg (f, V, y)$ is defined. By Proposition~\ref{Proposition51}(ii), there is an open neighborhood $W$ of $y,$ disjoint from $f(\partial V)$, on which $\deg (f, V, \cdot)$ is constant. Since $\dim (B_f \cup S) < n$ and since $f$ is continuous, one can  find a point $\bar{x} \in V \setminus (B_f \cup S)$ such that $f(\bar{x})\in W$. By the inverse mapping theorem (Theorem~\ref{InverseTheorem}), $f(V) \cap W$ must contain an open set $W'.$ We know that $\dim f(B_f \cup S)  < n,$ so there is some $\overline{y}$ in $W'$ that is not in $ f(B_f \cup S).$ Thanks to Theorem~\ref{DegreeTheorem}, we get
\begin{eqnarray*}
\deg (f, V, \overline{y}) &=& \sum_{x \in V \cap f^{-1}(\overline{y}) } \mathrm{sign} \det df(x).
\end{eqnarray*}
	Note that $V \cap f^{-1}(\overline{y})$ is nonempty and that $\det df(x)$ is nonzero for all $x \in V \cap f^{-1}(\overline{y}).$ Now the condition (i) implies that 
	$$\sum_{x \in V \cap f^{-1}(\overline{y}) } \mathrm{sign} \det df(x) \ne 0.$$
	Therefore $\deg (f, V, \overline{y}) \ne 0.$ Since $\deg (f, V, \cdot)$ is constant on $W,$ we have $\deg (f, V, w) \ne 0$ for all $w \in W.$ By the properties of degree, for each $w \in W,$ there exists $x \in V$ such that $f(x) = w.$ Consequently, $W \subset f(V) \subset f(U),$ which implies that $y$ lies in the interior of $f(U).$ Thus $f$ is quasi-interior.
\end{proof}

The following example, which is taken from \cite[Example~4.6]{Denkowski2017}, shows that the assumptions in Theorem~\ref{PourciauTheorem} do not imply that $f$ is open. (Recall that $f$ is {\em open} if it maps open sets onto open sets.)

\begin{example}
Consider the polynomial mapping 
$$f \colon \mathbb{R}^4 \to \mathbb{R}^4, \quad x := (x_1, x_2, x_3, x_4) \mapsto  (x_1^2 - x_2^2 - x_3^2 - x_4^2, 2x_1x_2, 2x_1x_3, 2x_1x_4).$$It is easy to see that $f$ is proper and that
\begin{eqnarray*}
df(x) = 
\begin{bmatrix}
2x_1 & -2x_2 & - 2x_3 & - 2x_4\\
2x_2 & 2x_1 & 0 &  0 \\
2x_3 & 0 & 2x_1 & 0 \\
2x_4 & 0 & 0 & 2x_1
\end{bmatrix}.
\end{eqnarray*}
Hence $\det df(x) = 16x_1^2 (x_1^2 + x_2^2 + x_3^2 + x_4^2) \ge 0$ for all $x \in \mathbb{R}^4.$ Moreover, the set of critical points of $f$ is
$\{0\} \times \mathbb{R}^3,$ which has measure zero. On the other hand since the set
$$f^{-1}{(-1, 0, 0, 0)} = \{x \in \mathbb{R}^4 \, : \, x_1 = 0, x_2^2 + x_3^2 + x_4^2 = 1 \}$$
is not finite, it follows from \cite[Theorem, page~297]{Gamboa1996} (see also \cite{Denkowski2017, Hirsch2002}) that $f$ is not open. 
\end{example}

Next, we are concerning with the global invertibility of continuous definable mappings. 
Let us  recall some facts. Let $f \colon \R^n\to\R^n$ be a $C^1$-mapping. By definition, then 
$$\nu_f(x) = \inf_{\|v\|=1}\|df(x)v\| \quad \textrm{ for all } \quad x \in \mathbb{R}^n.$$ 
Furthermore, we have $\nu_f(x) > 0$ if and only if the derivative mapping $df(x) \colon \mathbb{R}^n \to \mathbb{R}^n$ is surjective and in this case it is clear that $\nu_f(x) = \big \|[df(x)]^{-1} \big\|^{-1}.$ With these in mind, then a theorem of Hadamard \cite{Hadamard1906} asserts that $f$ is a global homeomorphism provided that
$$\nu_f(x) \ge \delta \quad \textrm{ for all } \quad x \in \mathbb{R}^n$$
for some scalar $\delta > 0.$ This result has a natural counterpart in our setting. In fact we shall prove a stronger statement as follows; see also \cite{Jelonek2003-2, Kurdyka2000-1, Pourciau1982, Rabier1997}.

\begin{theorem}\label{HadamardTheorem}
Let $f \colon \R^n\to\R^n$ be a definable and local homeomorphism such that the set 
$$K_\infty (f):=\{ y\in \R^n\colon \exists x_k \in \R^n, \|x_k \|\to\infty\  {\rm  such \ that }\ f(x_k )\to y\ {\rm and}\ \|x_k \|\nu_f(x_k )\to 0\}$$
is empty. Then $f$  is a global homeomorphism.
\end{theorem}

\begin{proof} 
Consider the set $J(f)$ of values  at which $f$ is not proper (cf. \cite{Jelonek1999}):
$$J_f := \{ y\in \R^n\colon \exists x_k\in \R^n, \|x_k\|\to\infty\  \mathrm{ such\ that }\ f(x_k)\to y\}.$$
Assume that we have proved (cf. \cite[Proposition~3.1]{Kurdyka2000-1}):
\begin{equation*}
J(f)=K_\infty (f).
\end{equation*}
This, of course, implies that the mapping $f$ is proper, and so $f$ is a global homeomorphism.
	
So we are left with proving the equality above. By definition, $K_\infty (f)\subset J(f)$ and so we need to prove that $J(f)\subset K_\infty (f).$ Suppose to the contrary that for some $y\in J(f)$ we have $y\notin K_\infty(f)$. Then there exist a scalar $c>0$ and  a sequence $x_k\in \R^n$ such that 
$$\|x_k\|\to\infty,  \quad f(x_k)\to y  \ \textrm{ and } \quad \|x_k\|\nu_f(x_k)\geq c.$$
In particular, for all $k,$ $x_k$ is a regular point of $f.$ In view of Lemma~\ref{Claim31}, we can assume that $x_k \not \in B_f$ for all $k.$ (Recall that the set $B_f$ of the points where $f$ is not of class $C^1$ is a nowhere dense definable subset of $\mathbb{R}^n.$) Using the curve selection lemma (Lemma~\ref{selectioncurve}), we can find a definable $C^1$-curve $\gamma \colon [r, +\infty) \to \R^n \setminus B_f$ ($r >0$) such that
$$\|\gamma(t)\| = t, \quad f(\gamma (t)) \to y\ {\rm as}\ t\to\infty\  \quad \textrm{ and } \quad \|\gamma (t)\|\nu_f(\gamma(t))\geq c.$$
For all $t \ge r$ we have $\gamma(t) \not \in B_f,$ and so $\nu_f(\gamma(t)) = \inf_{\|v\|=1}\|df(\gamma (t))v\| .$ It follows that
$$\inf_{\|v\|=1}\|df(\gamma (t))v\| \geq \frac{c}{\|\gamma(t)\|}=\frac{c}{t},$$
or equivalently,
$$\|df(\gamma (t))v\|\geq \frac{c}{t}\|v\| \quad {\rm for\ all } \quad v\in \R^n.$$
Taking $v := \gamma'(t)$ (here $\gamma'(t)$ is the derivative of the function $\gamma(\cdot)$), we obtain
$$\|df(\gamma (t))\gamma'(t)\|\geq \frac{c}{t}\|\gamma'(t)\|.$$
On the other hand, we have
$$2t = (t^2)' = \langle \gamma(t), \gamma(t) \rangle ' = 2\langle \gamma(t), \gamma'(t) \rangle,$$
which yields
$$t = |\langle \gamma(t), \gamma'(t) \rangle| \le \|\gamma(t)\| \cdot \|\gamma'(t)\| =  t \cdot \|\gamma'(t)\|.$$
Hence $\|\gamma'(t)\| \geq 1,$ and so
$$\|df(\gamma (t))\gamma'(t)\|\geq \frac{c}{t}.$$
Consequently, 
$$\int_r^\infty\| (f \circ \gamma)'(t)\|dt  = \int_r^\infty\|df(\gamma (t))\gamma'(t)\|dt \geq \int_r^\infty\frac{c}{t}dt=\infty.$$
Note that $f$ is injective on $\gamma$ for $r$ large enough (by Lemma~\ref{MonotonicityLemma}).  Therefore, the curve 
$$f \circ \gamma \colon [r, +\infty) \to \mathbb{R}^n, \quad t \mapsto (f \circ \gamma)(t),$$ 
has infinite length; or equivalently, $(f \circ \gamma)(t)$ has no limit in $\mathbb{R}^n,$ which is a contradiction to $f(\gamma (t)) \to y $ as $t \to +\infty.$
\end{proof}

We finish this section by providing a version of the {\em Hadamard integral condition} \cite{Hadamard1906} for continuous definable mappings; see \cite{Jaramillo2014, Pourciau1988} for related works.

\begin{theorem}
Let $f \colon \R^n\to\R^n$ be a continuous definable mapping. Assume that  there exists a locally Riemann integrable function $c\colon [0, +\infty) \to (0, +\infty)$ such that
\begin{equation*}
\int_0^\infty {c(t)}dt=\infty \quad \textrm{ and } \quad  \nu_f(x) \geq c(\|x\|)  > 0 \quad \textrm{ for all } \quad x \in \R^n.
\end{equation*}	
Then $f$ is a global homeomorphism. 
\end{theorem}

\begin{proof} 
By assumption, for every $x \in \R^n$ we have $\nu_f(x) > 0$ and so $x$ is a regular point of $f.$ In view of Theorem~\ref{InverseTheorem}, the mapping $f$ is a local homeomorphism. Hence it remains to show that $f$ is proper.

Suppose to the contrary that there exists a sequence $x_k \in \R^n$ such that 
$$\|x_k\|\to\infty \quad \textrm{ and } \quad f(x_k) \to y.$$
In view of Lemma~\ref{Lemma21}, $B_f$ is a definable nowhere dense subset of $\mathbb{R}^n.$ By Lemma~\ref{Claim31}, we can assume that $x_k \not \in B_f$ for all $k.$ Using the curve selection lemma (Lemma~\ref{selectioncurve}), we can find a definable $C^1$-curve $\gamma \colon [r, +\infty) \to \R^n \setminus B_f$ ($r >0$) such that
$$\|\gamma(t)\| = t \quad \textrm{ and } \quad f(\gamma (t)) \to y \ {\rm as}\ t \to \infty.$$
By assumption, we have 
$$\nu_f(\gamma(t)) \geq c(\|\gamma(t)\|).$$
Using the argument used in the proof of Theorem~\ref{HadamardTheorem}, we conclude that 
$$\|f'(\gamma (t))\gamma'(t)\|\geq {c(t)} \quad \textrm{ for all } \quad t \ge r,$$
which implies that
$$\int_r^\infty\| (f \circ \gamma)'(t)\|dt  = \int_r^\infty\|f'(\gamma (t))\gamma'(t)\|dt \geq \int_r^\infty{c(t)}dt = \infty.$$
So $(f \circ \gamma)(t)$ has no limit in $\mathbb{R}^n,$ which is a contradiction to $f(\gamma (t)) \to y $.
\end{proof}


\end{document}